\documentclass[a4paper,10pt,hidelinks]{amsart}
\usepackage{graphicx} 
\usepackage[a4paper, margin=2 cm]{geometry}
\usepackage{hyperref}
\usepackage{amsthm}
\usepackage{thm-restate}
\usepackage{geometry}
\usepackage{amssymb}
\usepackage{amsmath}
\usepackage{cleveref}
\usepackage{comment}
\usepackage{setspace}
\usepackage{xcolor}
\usepackage[shortlabels]{enumitem}
\usepackage{listings}
\usepackage{blkarray}
\usepackage{lmodern}
\usepackage[english]{babel}
\usepackage{float}
\usepackage{tikz}
\usetikzlibrary{arrows.meta,arrows}
\usetikzlibrary{arrows,decorations.markings}
\usetikzlibrary{decorations.pathmorphing}
\usepackage[square,sort,comma,numbers]{natbib}

\newtheorem{theorem}{Theorem}[section]
\newtheorem{lemma}[theorem]{Lemma}

\newtheorem{claim}[theorem]{Claim}
\newtheorem*{claim*}{Claim}
\newtheorem{proposition}[theorem]{Proposition}
\newtheorem{conjecture}[theorem]{Conjecture}

\newtheorem{problem}[theorem]{Problem}

\theoremstyle{definition}

\newtheorem*{definition*}{Definition}

\newcommand{\eps}{\varepsilon}

\renewcommand{\Pr}{\mathbb{P}}

\title{Chromatic number and regular subgraphs}

\author
{
Barnab\'{a}s Janzer
}
\author 
{
Raphael Steiner
}
\thanks{The second author gratefully acknowledges funding from the SNSF Ambizione Grant No. 216071. }
\author
{
Benny Sudakov 
}
\thanks{Research of the first and third authors was supported in part by SNSF grant 200021-228014.}
\address[Janzer, Steiner, Sudakov]{Department of Mathematics, Institute for Operations Research, ETH Z\"{u}rich, Switzerland.}
\email{\tt $\{$barnabas.janzer, raphaelmario.steiner, benjamin.sudakov$\}$@math.ethz.ch}

\date{}

\begin{document}

\begin{abstract}
    \setlength{\parskip}{\smallskipamount}
    \setlength{\parindent}{0pt}
    \noindent
In 1992, Erd\H{o}s and Hajnal posed the following natural problem: Does there exist, for every $r\in \mathbb{N}$, an integer $F(r)$ such that every graph with chromatic number at least $F(r)$ contains $r$ edge-disjoint cycles on the same vertex set? We solve this problem in a strong form, by showing that there exist $n$-vertex graphs with fractional chromatic number $\Omega\left(\frac{\log \log n}{\log \log \log n}\right)$ that do not even contain a $4$-regular subgraph. This implies that no such number $F(r)$ exists for $r\ge 2$. We show that assuming a conjecture of Harris, the bound on the fractional chromatic number in our result cannot be improved.
 \end{abstract}

\maketitle

\section{Introduction}
A great deal of attention in graph theory has been devoted to the following fundamental meta-question: Given a graph $G$ of huge chromatic number, what kind of subgraphs and substructures can we guarantee to find in it? Over the years, a rich and interconnected array of problems of this type has been discussed in the literature, see, e.g., the recent survey article~\cite{scott22} by Scott for an overview. Despite the popularity of these problems, many of them remain notoriously difficult to attack, in part because of a lack of tools that exist for finding structures in graphs with a large (but constant) chromatic number. 

Several natural problems of this type were raised by Erd\H{o}s and his collaborators, the most famous of which is a conjecture due to Erd\H{o}s and Hajnal~\cite{erdos1969}, stating that for all $k, g \in \mathbb{N}$ there exists an integer $f(k,g)$ such that every graph of chromatic number at least $f(k,g)$ has a subgraph of chromatic number $k$ and girth at least~$g$. The existence of $f(k,4)$ for all $k\in \mathbb{N}$ was proved using a beautiful argument by R\"{o}dl~\cite{roedl}, but progress has been lacking since. In this paper, we shall be concerned with the following problem raised by Erd\H{o}s and Hajnal~\cite[Problem~6]{erdos1992} dating back at least to 1992, that regards the structure of cycles in graphs of large chromatic number.

\begin{problem}[Erd\H{o}s and Hajnal 1992]\label{prob:2}
Is it true that for every $r\in \mathbb{N}$ there exists a number $F(r)$ such that every graph of chromatic number at least $F(r)$ contains $r$ edge-disjoint cycles on the same vertex set?
\end{problem}

This problem was repeated by Erd\H{o}s~\cite[Problem 12]{erdos1997} in a paper published in 1997, and is also included in the 1998 book~\cite{chung1998erdos} by Chung and Graham on Erd\H{o}s's legacy of problems, see also the problem entry on the corresponding website maintained by Fan Chung\footnote{See \url{https://mathweb.ucsd.edu/~erdosproblems/erdos/newproblems/ManyEdgeDisjointCycles.html}}. 

Since a graph with large chromatic number must contain many edges, Problem~\ref{prob:2} is closely related to another question of  Erd\H{o}s \cite{erdos1975}. In 1976 he asked for the maximum number of edges in an $n$-vertex graph that does not contain $r$ edge-disjoint cycles on the same vertex set. In a recent breakthrough on this problem, Chakraborti, Janzer, Methuku and Montgomery~\cite{chakraborti2024} proved an upper bound of the form $n \cdot \mathrm{polylog}(n)$ for every fixed $r$, improving the previous best bound of $O(n^{3/2})$. This directly implies (via degeneracy) that every $n$-vertex graph without $r$ edge-disjoint cycles on the same vertex set has chromatic number at most $\mathrm{polylog}(n)$. The best known lower bound on the number of edges in an $n$-vertex graph without $r$ edge-disjoint cycles on the same vertex set (for $r\ge 2$) is of the form $\Omega(n\cdot \log\log n)$. This follows from a famous construction due to Pyber, R\"{o}dl and Szemer\'{e}di~\cite{pyber1995} which is known at least since 1985 (see~\cite{pyber1985}). Concretely, the authors of~\cite{pyber1995} showed that there exist $n$-vertex graphs with $\Omega(n \log \log n)$ edges and no $k$-regular subgraph for any $k\ge 3$, thus in particular containing no edge-disjoint cycles with the same vertex set. However, this construction is inherently bipartite and can thus not directly be used to construct graphs of large chromatic number. In fact, perhaps this feature of Pyber, R\"{o}dl and Szemer\'{e}di's construction was what inspired Erd\H{o}s and Hajnal to pose Problem~\ref{prob:2} above.

As the main result of this paper, we give a strong negative answer to Problem~\ref{prob:2}, by showing the following result. Here, $\chi_f(G)$ denotes the well-known \emph{fractional chromatic number}.

\begin{theorem}\label{thm:main}
For some $c>0$ and all sufficiently large $n\in \mathbb{N}$, there exists an $n$-vertex graph $G$ such that $\chi_f(G)\ge c\frac{\log\log n}{\log \log \log n}$ and $G$ contains no $4$-regular subgraph. 
\end{theorem}

In particular, since $\chi(G)\ge \chi_f(G)$ for every graph $G$, this shows the existence of graphs with arbitrarily large chromatic number that do not even contain two edge-disjoint cycles with the same vertex set. Hence, none of the numbers $F(r)$ in the problem of Erd\H{o}s and Hajnal can exist for any $r$ greater than $1$ (trivially, we have $F(1)=3$). The proof of Theorem~\ref{thm:main} involves carefully analyzing a randomly constructed multi-partite variant of the construction of Pyber, Rödl and Szemerédi mentioned above. By slightly modifying our arguments we also obtain a much simpler proof for the existence of dense graphs without $3$-regular subgraphs (the original analysis by Pyber, Rödl and Szemer\'{e}di~\cite{pyber1995} involved heavy computations).

The problem of determining the asymptotic behaviour of the maximum number of edges of graphs without a $k$-regular subgraph (for some $k\ge 3$) was posed by Erd\H{o}s-Sauer \cite{ES75} in 1975 and attracted a lot of attention in the last 40 years. It was recently fully resolved by Janzer and Sudakov~\cite{janzersudakov2023}, who showed that the answer is $\Theta(n\log\log n)$, thus matching the lower-bound construction of Pyber, R\"{o}dl and Szemer\'{e}di.

\begin{theorem}[\cite{janzersudakov2023}]\label{thm:averagedegree}
For every integer $k\ge 3$ there exists some $C_k>0$ such that for sufficiently large $n$ every $n$-vertex graph without a $k$-regular subgraph has at most $C_kn\log\log n$ edges.
\end{theorem}

In the light of this discussion, it seems natural to ask for the asymptotic behaviour of the maximum chromatic number $g_k(n)$ of an $n$-vertex graph without a $k$-regular subgraph. Theorems~\ref{thm:main} and~\ref{thm:averagedegree} imply that $\Omega\left(\frac{\log \log n}{\log \log \log n}\right)\le g_k(n)\le O(\log \log n)$ for every fixed $k\ge 4$. Given the local sparsity of graphs without a regular subgraph, we suspect that the lower bound gives the truth for every $k\ge 3$. Supporting this claim, in Section~\ref{sec:frac} we will show that assuming a conjecture of Harris, every $n$-vertex graph without a $k$-regular subgraph has fractional chromatic number at most $O\left(\frac{\log \log n}{\log \log \log n}\right)$.

\textbf{Notation and Preliminaries.} Given a graph $G$, we denote by $V(G)$ its vertex set, by $E(G)$ its edge set, and by $e(G)$ its number of edges. For a subset $X\subseteq V(G)$, we denote by $G[X]$ the induced subgraph of $G$ with vertex set $X$, and for two disjoint sets $U, V\subseteq V(G)$ we denote by $e_G(U, V)$ the number of edges in $G$ with one endpoint in $U$ and one endpoint in $V$. Given a graph $G$, its \emph{fractional chromatic number} $\chi_f(G)$ is defined as the optimal value of the following linear program (by $\mathcal{I}(G)$ we denote the collection of independent sets in $G$):

\begin{align*}
    \text{min} \sum_{I \in \mathcal{I}(G)}&{x_I} \\
    \text{s.t.} \sum_{I\in \mathcal{I}(G): v \in I}&{x_I}\ge 1~~(\forall v \in V(G)), \\
    & x_I \ge 0~~(\forall I \in \mathcal{I}(G)).
\end{align*}

Note that the optimal value of the corresponding integer program (requiring $x_I\in \mathbb{Z}$) is exactly the chromatic number $\chi(G)$ of the graph, and thus $\chi(G)\ge \chi_f(G)$ for every graph $G$. By linear programming duality, we can also express $\chi_f(G)$ as the optimal value of

\begin{align*}
    \text{max} &\sum_{v\in V(G)}{w_v} \\
    \text{s.t.}~~~~&\sum_{v \in I}{w_v}\le 1~~(\forall I \in \mathcal{I}(G)), \\
    &~~~~~w_v \ge 0~~(\forall v\in V(G)).
\end{align*}

This implies that a graph $G$ satisfies $\chi_f(G)\le k$ for some real number $k>1$ if and only if for every weighting $w:V(G)\rightarrow \mathbb{R}_{\ge 0}$ there exists an independent set $I$ in $G$ such that $\sum_{v\in I}{w(v)}\ge \frac{1}{k}\sum_{v\in V(G)}{w(v)}$.

\section{Proof of Theorem~\ref{thm:main}}\label{sec:proof}

We define a random graph as follows. Set $\eps=\eps_n=\frac{1}{\sqrt{\log n}}$. Let $C=\frac{1}{10}\log\log n$, and take disjoint sets $B_1,\dots,B_C$ such that $|B_i|=n^{1-20^i\eps}$. (We will omit floors and ceilings.) The vertex set of our graph is $V=\bigcup_{i=1}^C B_i$ (and each $B_i$ will be an independent set). For each $i\in[C]$, each vertex $v\in B_i$, and each $j\in[C]$ with $j>i$, we independently and uniformly at random pick a vertex $w_{v,j}\in B_j$ to be the unique neighbour of $v$ in $B_j$. Then our random graph $G$ has edge set given by all such pairs $vw_{v,j}$. Note that $|V|<n$ for large $n$.

\begin{lemma}\label{lem:reg}
    With probability $1-o(1)$, $G$ does not contain a $4$-regular subgraph.
\end{lemma}
\begin{proof}
    Assume that $G$ has a $4$-regular subgraph $H$. Let $H$ have $s>0$ vertices, and, for convenience, set $B_{C+1}=\emptyset$ and $B_0=V$. Let $i\in [C]\cup\{0\}$ such that $|B_{i+1}|<s/1000\leq |B_{i}|$.

    \begin{claim}\label{claim:subgraph}
        For some $x>0$, $H$ has a subgraph $H''$ on $x$ vertices such that $V(H'')\subseteq \bigcup_{j=1}^{i-1}B_j$ and $e(H'')\geq 1.1x$.  
    \end{claim}
    \begin{proof}
        First note that $\sum_{j\geq i+1}|B_j|\leq |B_{i+1}|\sum_{j\geq i+1}n^{-(20^j-20^{i+1})\eps}\leq 2|B_{i+1}|\leq \frac{1}{500}s$. Since $H$ is $4$-regular, at most $\frac{1}{125}s$ edges of $H$ have an endpoint in $\bigcup_{j>i}B_j$. Let $H'$ denote the subgraph of $H$ induced by $V(H)\cap\bigcup_{j=1}^{i}B_j$, and $H''$ the subgraph induced by $V(H)\cap\bigcup_{j=1}^{i-1}B_j$. So our previous observations give $e(H')\geq (2-\frac{1}{100})s$.
        
        Let $X=V(H'')=V(H)\cap \bigcup_{j=1}^{i-1} B_{j}$ and $Y=V(H)\cap B_i$. Since $H$ is $4$-regular, we have $e_H(X,Y)\leq 4|Y|$. Also, since $H$ is a subgraph of $G$, we have $e_H(X,Y)\leq |X|$. It follows that $e_H(X,Y)\leq \frac{4}{5}(|X|+|Y|)\leq \frac{4}{5}s$, and hence $e(H'')\geq (2-\frac{1}{100})s-\frac{4}{5}s>1.1s\geq 1.1|X|.$ 
    \end{proof}
    
    For each $i\in [C]\cup\{0\}$, let $\mathcal{A}_i$ be the event that $G$ has a subgraph $H''$ on vertex set $V(H'')\subseteq \bigcup_{j=1}^{i-1}B_j$ such that for some $0<x\leq 1000|B_i|$ we have $|V(H'')|=x$ and $e(H'')\geq 1.1x$. Clearly, $\mathcal{A}_i$ can only hold if $i\geq 2$.
    Note that, if $n$ is large enough, for any fixed $i$ and $x$, the probability that $G$ has such a subgraph $H''$ is at most
    \begin{align*}
        \binom{n}{x}\binom{x^2/2}{\lceil 1.1x\rceil}\left(\frac{1}{|B_{i-1}|}\right)^{\lceil 1.1x\rceil}&\leq \left(\frac{en}{x}\right)^{x}\left(\frac{ex^2/2}{|B_{i-1}|\lceil 1.1x\rceil}\right)^{\lceil 1.1x\rceil}\\
        &\leq \left(\frac{en}{x}\right)^{x}\left(\frac{ex}{|B_{i-1}|}\right)^{1.1x}\\
        &\leq \left(10\frac{nx^{0.1}}{|B_{i-1}|^{1.1}}\right)^{x}\\
        &\leq \left(100\frac{n|B_i|^{0.1}}{|B_{i-1}|^{1.1}}\right)^{x}\\
        &= \left(100n^{-\frac{9}{10}20^{i-1}\eps}\right)^{x}\\
        &\leq e^{-\frac{1}{2}x\sqrt{\log n}}.
    \end{align*}
It follows that $\Pr[\mathcal{A}_i]\leq \sum_{x\geq 1} e^{-\frac{1}{2}x\sqrt{\log n}}\leq 2 e^{-\frac{1}{2}\sqrt{\log n}}$ (if $n$ is large enough). Hence, with probability $1-o(1)$, for all $i\in [C]\cup\{0\}$, $\mathcal{A}_i$ does not hold. The result follows using Claim~\ref{claim:subgraph}.
\end{proof}

\textbf{Remark.}
Our arguments can be used to give a significantly shorter proof of the result of Pyber, Rödl and Szemerédi~\cite{pyber1995} about the existence of graphs with $n$ vertices, $\Theta(n\log\log n)$ edges and no $3$-regular subgraphs -- the construction is (essentially) the same, but our calculations are much simpler.
Indeed, let $G'$ denote the subgraph of $G$ obtained by keeping only the edges which have an endpoint in $B_1$ and let $H$ be a $3$-regular subgraph of $G'$ with $s$ vertices and $1.5s$ edges.
Since $G'$ is bipartite and $H$ is regular we have exactly $s/2$ vertices of $H$ in $B_1$. Using the same notation as in the above proof, we have $e(H')\geq 1.49s$, $e_H(X,Y)\leq 3|Y|$ and $e_H(X,Y)\leq |X \cap B_1|=s/2$. If $X$ has size at most $0.9s$, then 
$e(H'')=e(H')-e_H(X,Y)\geq 0.99s \geq  1.1|X|$. Otherwise $|Y| \leq 0.1s$, $e_H(X,Y)\leq 0.3s$ and again $e(H'')=e(H')-e_H(X,Y)\geq 1.19s> 1.1|X|$. 
The rest of the proof works the same way as in the above. 

\begin{lemma}\label{lem:frac}
    With probability $1-o(1)$, $\chi_f(G)=\Omega(\frac{\log\log n}{\log\log\log n})$.
\end{lemma}
\begin{proof}
    We assign weight $1/|B_i|$ to each vertex in $B_i$. Since the total weight is $C=\frac{1}{10}\log\log n$, it suffices to show that, with probability $1-o(1)$, each independent set has weight at most $10\log C$.

    Let $\mathcal{E}_i$ be the event that $G$ contains an independent set $I\subseteq \bigcup_{j\geq i}B_j$ with total weight at least $9\log C$ such that $|I\cap B_i|/|B_i|\geq \frac{\log C}{C}$. Clearly, if $G$ has an independent set of weight at least $10\log C$, then $\mathcal{E}_i$ holds for some $i\in [C]$. Thus, we will upper bound the probability of $\mathcal{E}_i$ for each $i$.

    Given some $i$ and a positive integer $t\geq \frac{\log C}{C}|B_i|$, we consider the probability that there is some $I\subseteq \bigcup_{j\geq i}B_j$ with total weight at least $9\log C$ and $|I\cap B_i|=t$ such that $I$ is independent in $G$. Fix such a set $I$, and let $p_j=|I\cap B_j|/|B_j|$ for all $j\geq i$, and note that 
    $\frac{\log C}{C} \leq p_i\leq 1$. The probability that there are no edges in $G$ between $I\cap B_i$ and $I\cap \bigcup_{j>i}B_j$ is
    \begin{align*}
        \prod_{j>i} (1-p_j)^{p_i|B_i|}\leq e^{-p_i|B_i|\sum_{j>i} p_j}\leq e^{-p_i|B_i|\cdot 8\log C}.
    \end{align*}

Moreover, the number of such sets $I$ is at most
\begin{align*}
    \binom{|B_i|}{p_i|B_i|}2^{\sum_{j>i}|B_j|}.
\end{align*}
Note that $\sum_{j>i}|B_j|\leq 2|B_{i+1}|=2|B_i|n^{-(20^{i+1}-20^{i})\eps}\leq |B_i|e^{-5\sqrt{\log n}}$, and $\binom{|B_i|}{p_i|B_i|}\leq \left(\frac{e}{p_i}\right)^{p_i|B_i|}\leq e^{(1+\log(1/p_i))p_i|B_i|}$.
Thus, by the union bound, the probability that such an independent set exists in $G$ is at most
\begin{align*}
    e^{-p_i|B_i|\cdot 8\log C}\cdot e^{(1+\log(1/p_i))p_i|B_i|}\cdot 2^{|B_i|e^{-5\sqrt{\log n}}}&\leq \exp\left({|B_i|\left(-8p_i\log C+p_i(1+\log(1/p_i))+e^{-5\sqrt{\log n}}\right)}\right)\\
    &\leq \exp\left({|B_i|\left(-8p_i\log C+2p_i\log C+e^{-5\sqrt{\log n}}\right)}\right)\\
    &\leq \exp\left({-|B_i|\left(6(\log C)^2/C-e^{-5\sqrt {\log n}}\right)}\right)\\
    &=\exp\left({-\Omega(|B_i|/\log\log n)}\right).
\end{align*}

The calculation above shows that for any fixed value of $|I\cap B_i|$, the probability of such an $I$ existing is $e^{-\Omega(|B_i|/\log\log n)}$. Since there are at most $|B_i|$ choices for the value of $|I\cap B_i|$, this gives
\[\Pr[\mathcal{E}_i]\leq |B_i|e^{-\Omega(|B_i|/\log\log n)}\leq e^{-n^{1/2}}\]
if $n$ is large enough. Hence, with probability $1-o(1)$, none of the events $\mathcal{E}_i$ (for $i\in [C]$) hold, and hence $G$ has no independent set of weight at least $10\log C$.
\end{proof}

Theorem~\ref{thm:main} now follows by combining Lemma~\ref{lem:reg} and Lemma~\ref{lem:frac}, noting that for sufficiently large $n$ the constructed graph $G$ has $<n$ vertices. We can then simply fill up $G$ with isolated vertices to make the number of vertices match exactly $n$, while leaving the fractional chromatic number unchanged.

\textbf{Remark.} One can show that the graph $G$ constructed above indeed has chromatic number $\chi(G)=\Theta\left(\frac{\log\log n}{\log \log \log n}\right)$, i.e., the lower bound on $\chi(G)$ given by the fractional chromatic number $\chi_f(G)$ is tight up to a constant factor.

\section{Fractional chromatic number of graphs without regular subgraphs}\label{sec:frac}

In this section, we show that assuming the following conjecture of Harris about the fractional chromatic number of triangle-free graphs, the lower bound on the fractional chromatic number in our main result, Theorem~\ref{thm:main}, cannot be improved. 

\begin{conjecture}[\cite{harris2019}]\label{con:harris}
There exists an absolute constant $K>0$ such that for every sufficiently large $d\in \mathbb{N}$, every triangle-free  $d$-degenerate graph $G$ satisfies $\chi_f(G)\le K\frac{d}{\log d}$. 
\end{conjecture}

We remark that Harris' conjecture has received significant attention in recent years, see e.g.~\cite{blumenthal22,kelly2024,kwan2020}. 
We can now show the claimed result, whose proof is based on a subsampling-trick.

\begin{proposition}
If Conjecture~\ref{con:harris} holds, then for every $k\in \mathbb{N}$ there exists a constant $C_k>0$ such that for sufficiently large $n\in \mathbb{N}$ every $n$-vertex graph $G$ without a $k$-regular subgraph satisfies 
$$\chi_f(G)\le C_k\frac{\log \log n}{\log \log \log n}.$$
\end{proposition}
\begin{proof}
Fix $k\in \mathbb{N}$. By Theorem~\ref{thm:averagedegree} there exists a constant $C>0$ such that every sufficiently large $n$-vertex graph with average degree at least $C\log\log n$ contains a $k$-regular subgraph. This implies that there exists a constant $C'\ge 1$ such that every sufficiently large $n$-vertex graph $G$ without a $k$-regular subgraph is $\lfloor C'\log\log n\rfloor$-degenerate. 

Let $C'':=24KC'$, where $K$ is the constant from Conjecture~\ref{con:harris}. We claim that every sufficiently large $n$-vertex graph $G$ without a $k$-regular subgraph satisfies $\chi_f(G)\le C''\frac{\log \log n}{\log \log \log n}$. To do so, by definition, it suffices to show that for every weighting $w:V(G)\rightarrow \mathbb{R}_{\ge 0}$ with $\sum_{v\in V(G)}{w(v)}=1$, there exists an independent set $I$ in $G$ such that $\sum_{v\in I}{w(I)}\ge \frac{\log \log \log n}{C''\log \log n}$. 

To find such an independent set, we first show the following.
\begin{claim*}
There exists a subset $X\subseteq V(G)$ with $\sum_{v\in X}{w(v)}\ge \frac{1}{3}(\log\log n)^{-3/4}$ such that $G[X]$ is triangle-free and $\lfloor 2C'(\log \log n)^{1/4} \rfloor$-degenerate.
\end{claim*}
\begin{proof}[Proof of the Claim.]
Let $p:=(\log \log n)^{-3/4} \in [0,1]$. Let $v_1,\ldots,v_n$ be a linear ordering of $V(G)$ witnessing the degeneracy of $G$, i.e., such that $v_i$ has at most $C'\log \log n$ neighbors among $\{v_1,\ldots,v_{i-1}\}$ for every $i \in [n]$. Now, consider the following random process to generate $X$: First, create a random subset $Y$ of $V(G)$ in which every vertex is included independently with probability $p$. Then, create a subset $X\subseteq Y$ according to the following rule: A vertex $v_i\in Y$ is included in $X$ if and only if $|N(v_i)\cap \{v_1,\ldots,v_{i-1}\}\cap Y|\le 2C'(\log\log n)^{1/4}$ and $N(v_i)\cap \{v_1,\ldots,v_{i-1}\}\cap Y$ is an independent set in $G$. 

It follows easily from the definition of this process that $G[X]$ is $\lfloor 2C'(\log \log n)^{1/4}\rfloor$-degenerate and triangle-free. So, to prove the claim it suffices to show that $\mathbb{E}\left[\sum_{v\in X}{w(v)}\right]\ge \frac{1}{3}(\log \log n)^{-3/4}$. By linearity of expectation, it is enough to show that $\mathbb{P}[v_i \in X]\ge \frac{1}{3}(\log \log n)^{-3/4}$ for every $i \in [n]$. Note that by Markov's inequality, for every $i\in [n]$, we have that 
$$\mathbb{P}[|N(v_i)\cap \{v_1,\ldots,v_{i-1}\}\cap Y|>2C'(\log \log n)^{1/4}]\le \frac{pC'\log\log n}{2C'(\log \log n)^{1/4}}=\frac{1}{2}$$ and 
$$\mathbb{P}[N(v_i)\cap \{v_1,\ldots,v_{i-1}\}\cap Y\text{ not independent}]\le p^2|E(G[N(v_i)\cap \{v_1,\ldots,v_{i-1}\}])|.$$
Since $G[N(v_i)\cap \{v_1,\ldots,v_{i-1}\}]$ is a graph on at most $C'\log \log n$ vertices without a $k$-regular subgraph, for $n$ large enough Theorem~\ref{thm:averagedegree} implies that $|E(G[N(v_i)\cap \{v_1,\ldots,v_{i-1}\}])| \le O(\log \log n\log\log\log\log n)\le \frac{1}{6}(\log \log n)^{3/2}$. Hence, we obtain $$\mathbb{P}[N(v_i)\cap \{v_1,\ldots,v_{i-1}\}\cap Y\text{ not independent}]\le \frac{1}{6}p^2(\log \log n)^{3/2}\le \frac{1}{6}.$$
Altogether, this implies (by definition of the process) that $\mathbb{P}[v_i \in X]\ge \mathbb{P}[v_i\in Y]\left(1-\frac{1}{2}-\frac{1}{6}\right)=\frac{1}{3}p=\frac{1}{3}(\log \log n)^{-3/4}$. This shows that there exists a subset $X\subseteq V(G)$ with the required properties, which concludes the proof of the claim.
\end{proof}
Since $G[X]$ is triangle-free and $d$-degenerate for $d=\lfloor 2C'(\log \log n)^{1/4}\rfloor$, we may now apply the statement of Conjecture~\ref{con:harris} to $G[X]$ to find that $\chi_f(G[X])\le K\frac{d}{\log d}$. In particular, this means that there exists an independent set $I\subseteq X$ in $G$ such that $$\sum_{v\in I}{w(v)}\ge \frac{\log d}{Kd}\sum_{v\in X}{w(v)}\ge \frac{\log d}{Kd}\frac{1}{3}(\log \log n)^{-3/4}\ge \frac{\log\log \log n}{24KC'\log \log n}=\frac{\log \log \log n}{C''\log \log n}.$$ This is the desired statement, and hence $\chi_f(G)\le C''\frac{\log \log n}{\log \log \log n}$, as initially claimed. 
\end{proof}

\textbf{Acknowledgments.} We would like to thank Oliver Janzer for interesting discussions related to the topic.
\vspace{-0.3cm}
\bibliographystyle{abbrv}
\bibliography{bibliography.bib}

\begin{thebibliography}{10}

\bibitem{blumenthal22}
A.~Blumenthal, B.~Lidick\'y, R.~R. Martin, S.~Norin, F.~Pfender, and J.~Volec.
\newblock Counterexamples to a conjecture of {H}arris on {H}all ratio.
\newblock {\em SIAM Journal on Discrete Mathematics}, 36(3):1678--1686, 2022.

\bibitem{chakraborti2024}
D.~Chakraborti, O.~Janzer, A.~Methuku, and R.~Montgomery.
\newblock Edge-disjoint cycles with the same vertex set, 2024.
\newblock arXiv preprint, arXiv:2404.07190.

\bibitem{chung1998erdos}
F.~Chung and R.~Graham.
\newblock {\em Erd{\H{o}}s on graphs: His legacy of unsolved problems}.
\newblock AK Peters/CRC Press, 1998.

\bibitem{erdos1969}
P.~Erd\H{o}s.
\newblock Problems and results in combinatorial analysis and graph theory.
\newblock In {\em Proof techniques in graph theory}, pages 27--35. Academic Press, New York, 1969.

\bibitem{erdos1975}
P.~Erd\H{o}s.
\newblock Problems and results in graph theory and combinatorial analysis.
\newblock In {\em Proc. British Comb. Conf.}, number~5, pages 169--192, 1976.

\bibitem{ES75}
P.~Erd{\H{o}}s.
\newblock Some recent progress on extremal problems in graph theory.
\newblock {\em Congr. Numer.}, 14:3--14, 1975.

\bibitem{erdos1992}
P.~Erd{\H{o}}s.
\newblock Some of my favourite problems in various branches of combinatorics.
\newblock {\em Le Matematiche}, 47(2):231--240, 1992.

\bibitem{erdos1997}
P.~Erd{\H{o}}s.
\newblock Some recent problems and results in graph theory.
\newblock {\em Discrete Mathematics}, 164(1-3):81--85, 1997.

\bibitem{harris2019}
D.~G. Harris.
\newblock Some results on chromatic number as a function of triangle count.
\newblock {\em SIAM Journal on Discrete Mathematics}, 33(1):546--563, 2019.

\bibitem{janzersudakov2023}
O.~Janzer and B.~Sudakov.
\newblock Resolution of the {E}rd{\H{o}}s--{S}auer problem on regular subgraphs.
\newblock {\em Forum of Mathematics, Pi}, 11:e19, 2023.

\bibitem{kelly2024}
T.~Kelly and L.~Postle.
\newblock Fractional coloring with local demands and applications to degree-sequence bounds on the independence number.
\newblock {\em Journal of Combinatorial Theory, Series B}, 169:298--337, 2024.

\bibitem{kwan2020}
M.~Kwan, S.~Letzter, B.~Sudakov, and T.~Tran.
\newblock Dense induced bipartite subgraphs in triangle-free graphs.
\newblock {\em Combinatorica}, 40(2):283--305, 2020.

\bibitem{pyber1985}
L.~Pyber.
\newblock Regular subgraphs of dense graphs.
\newblock {\em Combinatorica}, 5(4):347--349, 1985.

\bibitem{pyber1995}
L.~Pyber, V.~R\"odl, and E.~Szemer{\'e}di.
\newblock Dense graphs without 3-regular subgraphs.
\newblock {\em Journal of Combinatorial Theory, Series B}, 63(1):41--54, 1995.

\bibitem{roedl}
V.~R\"{o}dl.
\newblock On the chromatic number of subgraphs of a given graph.
\newblock {\em Proceedings of the American Mathematical Society}, 64(2):370--371, 1977.

\bibitem{scott22}
A.~Scott.
\newblock Graphs of large chromatic number.
\newblock In {\em Proceedings of the International Congress of Mathematicians}, volume~6, pages 4660--4681, 2022.

\end{thebibliography}
\end{document}